\documentclass[12pt,a4paper]{article}

\usepackage{graphicx}
\usepackage[a4paper, hmargin={2.5cm,2.5cm},vmargin={3.3cm,3.3cm}]{geometry}

\usepackage[T1]{fontenc}
\usepackage[utf8]{inputenc}
\usepackage[english]{babel}

\usepackage{amsmath,amssymb,amsfonts}
\usepackage{amsthm,latexsym,xcolor}

\usepackage{cite}
\usepackage{paralist}
\usepackage{url}
\usepackage[bf,compact,pagestyles,small]{titlesec}
\titlespacing*{\section}{0pt}{14pt}{4pt}
\titlespacing*{\subsection}{0pt}{8pt}{3pt}

\usepackage{mathdots}

\usepackage{fancyhdr,lastpage}
\def\maketimestamp{\count255=\time
\divide\count255 by 60\relax
\edef\thetime{\the\count255:}%
\multiply\count255 by-60\relax
\advance\count255 by\time
\edef\thetime{\thetime\ifnum\count255<10 0\fi\the\count255}
\edef\thedate{\number\day-\ifcase\month\or Jan\or Feb\or Mar\or
             Apr\or May\or Jun\or Jul\or Aug\or Sep\or Oct\or
             Nov\or Dec\fi-\number\year}
\def\timstamp{\hbox to\hsize{\tt\hfil\thedate\hfil\thetime\hfil}}}
\maketimestamp
\fancypagestyle{plain}{%
\fancyhf{} 
\fancyfoot[C]{\small \thepage{} of \pageref{LastPage}}}

\numberwithin{equation}{section}  
\allowdisplaybreaks[4]

\newtheorem{theorem}{Theorem}[section]
\newtheorem{conjecture}{Conjecture}

\newtheorem{lemma}[theorem]{Lemma}
\newtheorem{proposition}[theorem]{Proposition}
\newtheorem{corollary}[theorem]{Corollary}

\theoremstyle{definition}
\newtheorem{definition}[theorem]{Definition} 

\theoremstyle{remark}
\newtheorem{remark}{Remark}

\DeclareMathOperator{\supp}{supp} %
\DeclareMathOperator{\ft}{\mathcal{F}}

\DeclareMathOperator{\identity}{Id}

\newcommand{\dila}[1][A^j]{D_{\! #1}} 
\newcommand{\id}[1][]{\identity_{#1}}

\newcommand{\meas}[1]{\abs{#1}}
\newcommand{\card}[1]{\#\vert #1 \vert}
\newcommand{\cardbig}[1]{\#\bigl\vert #1 \bigr\vert}

\newcommand{\cardpropbig}[2]{\# \bigl\vert {#1}:{#2} \bigr\vert}

\newcommand{\eps}{\ensuremath{\varepsilon}}

\newcommand*{\numbersys}[1]{\ensuremath{\mathbb{#1}}}
\newcommand*{\B}{\mathbf{B}}
\newcommand*{\C}{\numbersys{C}}
\newcommand*{\R}{\numbersys{R}}
\newcommand*{\Rn}{\numbersys{R}^n}
\newcommand*{\Q}{\numbersys{Q}}

\newcommand*{\Z}{\numbersys{Z}}
\newcommand*{\Zn}{\numbersys{Z}^n}
\newcommand*{\N}{\numbersys{N}}

\newcommand{\abs}[1]{\ensuremath{\left\lvert#1\right\rvert}}
\newcommand{\abssmall}[1]{\ensuremath{\lvert#1\rvert}}

\newcommand{\norm}[2][]{\ensuremath{\left\lVert#2\right\rVert_{#1}}}
\newcommand{\enorm}[2][]{\ensuremath{\abs{#2}_{#1}}}

\newcommand{\normsmall}[2][]{\ensuremath{\lVert#2\rVert_{#1}}}
\newcommand{\innerprod}[3][]{\ensuremath{\left\langle #2,#3\right\rangle_{\! #1}}}

\newcommand{\set}[1]{\ensuremath{\left\lbrace{#1}\right\rbrace}}
\newcommand{\seq}[1]{\ensuremath{\left\lbrace{#1}\right\rbrace}}
\newcommand{\setprop}[2]{\ensuremath{\left\lbrace{#1} : {#2}\right\rbrace}}

\newcommand{\af}[1]{\ensuremath{\mathcal{A}\left(#1\right)}}

\newcommand{\lat}[1]{\ensuremath{#1}} 
 \newcommand{\LG}{\ensuremath\lat{\Gamma}}
 
 \newcommand{\mlat}{k}
 
\makeatletter
\renewcommand*\env@matrix[1][c]{\hskip -\arraycolsep
  \let\@ifnextchar\new@ifnextchar
  \array{*\c@MaxMatrixCols #1}}
\makeatother
\newlength{\bracewidth}

\usepackage{xspace}
\newcommand{\ie}{i.e.,\xspace} 

\newcommand{\spa}{\operatorname{span}}

\newcommand{\ch}{\mathbf 1}

\newcommand{\ga}{\gamma}
\newcommand{\Ga}{\Gamma}
\newcommand{\ve}{\varepsilon}

\newcommand{\la}{\lambda}
\newcommand{\La}{\Lambda}

\newcommand{\Nm}{\operatorname{Nm}}

\usepackage{hyperref}
\hypersetup{
 pdfview={FitH},
 pdfstartview={FitH},
 bookmarks={true},
 pdftoolbar={false},
 pdfmenubar={false},
 pdfauthor = {Marcin Bownik, Jakob Lemvig}, 
 pdftitle = {Wavelets for non-expanding dilations}, 
 pdfsubject = {wavelet theory},
 pdfkeywords = {non-expasive dilations},
 pdfcreator = {LaTeX with hyperref package},
 pdfproducer = {PDFlatex}}

\makeatletter
\def\blfootnote{\xdef\@thefnmark{}\@footnotetext} 
\def\subjclass{\xdef\@thefnmark{}\@footnotetext}
\long\def\symbolfootnote[#1]#2{\begingroup%
\def\thefootnote{\fnsymbol{footnote}}\footnote[#1]{#2}\endgroup} 
\if@titlepage
  \renewenvironment{abstract}{%
      \titlepage
      \null\vfil
      \@beginparpenalty\@lowpenalty
      \begin{center}%
        \bfseries \abstractname
        \@endparpenalty\@M
      \end{center}}%
     {\par\vfil\null\endtitlepage}
\else
  \renewenvironment{abstract}{%
      \if@twocolumn
        \section*{\abstractname}%
      \else
        \small
        \list{}{%
          \settowidth{\labelwidth}{\textbf{\abstractname:}}
          \setlength{\leftmargin}{50pt}
          \setlength{\rightmargin}{50pt}
          \setlength{\itemindent}{\labelwidth}
          \addtolength{\itemindent}{\labelsep}
        }
        \item[\textbf{\abstractname:}]

      \fi}
      {\if@twocolumn\else\endlist\fi}
\fi
\makeatother

\begin{document}
\title{Wavelets for non-expanding dilations 
\\
and the lattice counting estimate}

\author{Marcin Bownik\footnote{University of Oregon, Department of Mathematics, Eugene, OR 97403--1222, USA, E-mail: \url{mbownik@uoregon.edu}}\phantom{$\ast$}, Jakob Lemvig\footnote{Technical University of Denmark, Department of Applied Mathematics and Computer Science, Matematiktorvet 303, 2800 Kgs. Lyngby, Denmark, E-mail: \url{jakle@dtu.dk}}}

\date{\today}

\maketitle 

\thispagestyle{plain}
\begin{abstract} We show that problems of existence and characterization of
  wavelets for non-expanding dilations are intimately connected with the
  geometry of numbers; more specifically, with a bound on the number of lattice
  points in balls dilated by the powers of a dilation matrix $A \in
  \mathrm{GL}(n,\R)$. This connection is not visible for the well-studied class
  of expanding dilations since the desired lattice counting estimate holds
  automatically. 
  We show that the lattice counting estimate holds for all dilations $A$ with
  $\abs{\det{A}}\ne 1$ and for almost every lattice $\Gamma$ with respect to the
  invariant probability measure on the set of lattices. As a consequence, we
  deduce the existence of minimally supported frequency (MSF) wavelets
  associated with such dilations for almost every choice of a lattice. Likewise,
  we show that MSF wavelets exist for all lattices and and almost every choice
  of a dilation $A$ with respect to the Haar measure on $\mathrm{GL}(n,\R)$.
\end{abstract}

\maketitle

\section{Introduction}
\label{sec:introduction}

A wavelet system is a collection of dilates and translates of a
function $\psi \in L^2(\R^n)$ given by
$\{\abs{\det{A}}^{j/2}\psi(A^j\cdot-\gamma)\}_{j\in \Z,\gamma \in
  \LG}$, where $A$ is an invertible $n \times n$ real matrix and $\LG$ is a
full rank lattice in $\R^n$.  The study of wavelets in higher dimensions
is generally restricted to the class of expanding dilations $A$, often
additionally assumed to preserve the integer lattice, $A\Z^n \subset
\Z^n$. Recall that a real $n\times n$ matrix is expanding, or
expansive, if all of its eigenvalues $\lambda$ satisfy
$|\lambda|>1$. This is due to the fact that many classical
results initially established for dyadic dilations $A=2\id$, first in
dimension $n=1$, and then in higher dimensions, often extend to the
setting of expanding dilations. This includes existence of several
classes of wavelets: well-localized wavelets in time and frequency,
minimally supported frequency (MSF) wavelets, Haar-type wavelets, and
Parseval wavelet frames. For example, Dai, Larson, and Speegle
\cite{DLS1, DLS2} have shown the existence of MSF wavelets for all
expanding dilations with real coefficients. In addition, wavelet
expansions associated with expanding dilations characterize many
classical function spaces such as: Lebesgue, Hardy, Lipschitz,
Sobolev, Besov, and Triebel-Lizorkin spaces.

In contrast, much less attention has been devoted to the study of wavelets
associated with general invertible dilations. Speegle in his thesis raised the
problem of existence of MSF wavelets for non-expanding dilations, and the first
example of a wavelet of this kind appeared in \cite{BS}. Laugesen \cite{La} and
Hern\'andez,  Labate, and Weiss \cite{HLW} then initiated a systematic study of
wavelets $\psi \in L^2(\R^n)$ in two distinct settings: amplifying dilations for
$\psi$ and dilations expanding on a subspace, respectively. In particular,
Hern\'andez et al.\ \cite{HLW} introduced an important concept, known as the
local integrability condition, that yields characterization results for Parseval
wavelet frames for non-expanding dilations, see \cite{MR2264324, Ku}. Soon
after, Speegle \cite{Sp} achieved breakthrough results giving necessary and
sufficient conditions for the existence of MSF wavelets for non-expanding
dilations. Based on Speegle's work, Ionascu and Wang \cite{IW} proved a beautiful result that gives a complete characterization of dilations admitting MSF wavelet in the dimension $n=2$. The corresponding problem in higher dimensions $n \ge 3$ remains open.

In this paper we show that problems of existence and characterization of
wavelets for non-expanding dilations are intimately connected with the geometry
of numbers and, more specifically, with the problem of bounding the number of
lattice points lying inside balls dilated by powers of a dilation matrix
$A$. The existence of a link between wavelets for non-expanding dilations and
diophantine approximation was already manifested in the papers of Speegle
\cite{Sp} and Ionascu and Wang \cite{IW}. However, this link is completely
invisible in the standard setting of expanding dilations, where the desired \emph{lattice counting estimate}, see Definition~\ref{def:lattice-counting}, holds automatically.

\begin{definition}
\label{def:lattice-counting}
Suppose $A$ is an $n\times n$ invertible matrix such that $\abs{\det{A}}>1$. 
We say that a pair $(A,\LG)$ satisfies
  the \emph{lattice counting estimate} if
\begin{equation}\label{eq:lc}
    \cardbig{ \LG \cap A^j(\B(0,r)) } \le C \max (1, \abs{\det{A}}^j)
    \qquad\text{for all }j\in \Z,
  \end{equation}
where $\B(0,r)$ denotes the open ball of radius $r>0$ centered at $0$. 
\end{definition}

We remark that if \eqref{eq:lc} holds for some $r=r_0>0$, then it holds for all $r>0$. This can be deduced from Lemma~\ref{thm:prog}, which guarantees the existence of large arithmetic progressions in the intersection of a lattice with a symmetric convex body. Of course, the constant $C$ in \eqref{eq:lc} will depend on $r$. 

In the context of wavelets we shall consider the lattice counting estimate in Fourier domain for the transpose dilation $B=A^T$ and the dual lattice
\[
\Ga^* = \{ x\in \R^n: \langle x, y \rangle \in \Z \quad\text{for }y\in \Ga\}
\]
that takes the form
\begin{equation}\label{eq:lca}
    \cardbig{ \LG^* \cap B^j(\B(0,r)) } \le C \max (1, \abs{\det{B}}^j)
\qquad\text{for all }j\in \Z.
  \end{equation}
  
We show that the lattice counting estimate \eqref{eq:lca}
characterizes the pairs of dilations and lattices $(B,\Gamma^*)$ for
which a rather technical local integrability condition is actually
equivalent with the much less technical integrability of the Calder\'on sum
$
\sum_{j\in \Z} |\hat \psi(B^{-j} \xi)|^2
$,
that is known to plays a key role in characterization of frame wavelets.

We also show that the lattice counting estimate holds not only for expanding
dilations, including those expanding on a subspace, it is even ubiquitous in a
probabilistic sense. That is, for any dilation $A$ with $\abs{\det{A}}>1$,
almost any choice of lattice $\Ga$ yields the lattice counting estimate. It also
holds for any fixed lattice $\Ga$ and almost every choice of a dilation
$A$. These results are shown using techniques introduced by Skriganov
\cite{Sk2} in his study of the logarithmically small errors in the lattice
problem for polyhedra. In particular, our arguments rely on diophantine
characteristic of a lattice, introduced by Skriganov \cite{Sk2}, and on several result in the geometry of numbers on intersection of convex bodies with lattices.

An interesting consequence of our ubiquity results is the existence of MSF wavelets for almost all random choices of dilations and lattices $(A,\Ga)$. That is, for any fixed lattice $\Ga$, which by standard arguments reduces to the key case $\Ga=\Z^n$, there exists an MSF wavelet for almost every choice of a dilation $A\in \mathrm{GL}(n,\R)$. Likewise, for any choice of a dilation $A\in \mathrm{GL}(n,\R)$, outside the exceptional case $\abs{\det{A}}=1$ for which MSF wavelets do not exist by the work of Larson, Schulz, Speegle, and Taylor \cite{MR2249311}, almost every (with respect to appropriate invariant measure on the set of all lattice) choice of a lattice $\Ga$ yields an MSF wavelet. Hence, MSF wavelets exists not only for all expanding dilations as was shown in \cite{DLS1}, but also for all invertible dilations $A$ and a generic choice of a lattice $\Ga$. Consequently, the pairs $(A,\Ga)$ that do not admit MSF wavelets form a thin and rather pathological exceptional set which is challenging to characterize beyond the known case \cite{IW} of the dimension $n=2$.

\section{Dilations expanding on a subspace}
\label{sec:expanding-subspace}

In this section we investigate the properties of the class of dilations that are \emph{expanding on a subspace}, that were introduced by Hern\'andez,  Labate, and Weiss \cite{HLW}. For this class of dilations wavelet characterization results, such as the characterization of Parseval wavelet frames, are known to hold, see \cite[Theorem 5.3]{HLW} and \cite[Theorem~1.1]{MR2264324}. Note that Guo and Labate \cite{MR2264324} corrected an error in the proof of
the characterization result from \cite{HLW} by redefining the class of 
dilation matrices expanding on a subspace. We give an explicit characterization of dilations that are expanding on a subspace. We also show that they correspond exactly to those dilations $A$ that satisfy the lattice counting estimate \eqref{eq:lc} for all possible choices of a lattice $\Ga$. 

Following Guo and Labate \cite{MR2264324} we adopt the following definition.

\begin{definition}
\label{def:exp-on-subspace}
  Given $A \in \mathrm{GL}(n,\R)$ and a non-zero linear subspace $F \subset \R^n$, we say that $A$ is expanding on $F$ if there exists a complementary (not necessarily orthogonal) linear
subspace $E$ of $\R^n$ with the following properties:
\begin{enumerate}[(i)]
\item $\R^n = F + E$ and $F \cap E = \{0\}$,
\item $F$ and $E$ are invariant under $A$, that is, $A(F) = F$ and $A(E) = E$, 
\item $\exists c \ge 1 \, \exists \gamma > 1 \, \forall j \ge 0: \enorm{A^jx} \ge (1/c) \gamma^j \enorm{x}$ for all $x \in F$,
\item  $\exists k>0 \, \forall j \ge 0: \enorm{A^jx} \ge k \enorm{x}$ for all $x \in E$.
\end{enumerate}
\end{definition}

\begin{remark}
\label{rem:E-and-F-from-exp-on-subs}
  Note that if $A$ expanding on a subspace, then all eigenvalues
  satisfy $\abs{\lambda}\ge 1$. Indeed, eigenvalues $\lambda$ of
  $A|_F$ must satisfy $|\lambda|>1$, whereas eigenvalues $\lambda$ of
  $A|_{E}$ satisfy $|\lambda| \ge 1$. Hence, we can take $E$ to be the
  (real) eigenspace associated with eigenvalues of modulus one; here
  we take the real and imaginary parts of eigenvectors associated with
  a complex conjugate pair of eigenvalues.
\end{remark}

Since $E$ is invariant under $A$, condition (iv) in
Definition~\ref{def:exp-on-subspace} is equivalent to the existence of $k>0$
such that, for all $j \ge 0$, we have $\enorm{x} \ge k \enorm{A^{-j} x}$ for all $x \in E$. This is equivalent to saying that the discrete time mapping $x \mapsto A^{-1}x, E \to E$, has a Lyapunov stable (sometimes called a marginally stable) fixed point at $x=0$. It is well-known that this discrete time mapping is Lyapunov stable if and only if all eigenvalues of $A^{-1}$ are no greater than one, and eigenvalues of modulus one have Jordan blocks of order one, \ie the algebraic and geometric multiplicity agree. We thereby obtain a simple characterization of the class of dilation matrices that are expansive on a subspace.  

\begin{proposition}
\label{thm:exp-subspace-charac}
  Let $A \in \mathrm{GL}(n,\R)$ be given. Then $A$ is expanding on a subspace if and only if
\begin{enumerate}[(i)]
\item all eigenvalues of $A$ have modulus greater than or equal to $1$, and
\item at least one eigenvalue has modulus strictly greater than $1$, and
\item all eigenvalues of modulus equal to $1$ have Jordan blocks of order one. 
\end{enumerate}  
\end{proposition}

\begin{proof}
 To prove the ``only if''-direction, assume towards a contradiction that the eigenvalue $\lambda$ of $A^{-1}$, $\abs{\lambda}=1$, has an algebraic multiplicity strictly greater than its geometric multiplicity. Let $v_1$ be an eigenvector and $v_2$ a generalized eigenvector of $A^{-1}$ associated with $\lambda$ such that
\[ 
A^{-1}v_1 = \lambda v_1 \qquad \text{and} \qquad A^{-1} v_2=\lambda v_2 + v_1.
\] 
Assume that $\lambda$ is non-real; the case $\lambda=\pm 1$ can be handled similarly. Then $\overline{\lambda}$ is also an eigenvalue with eigenvector $\overline{v_1}$ and generalized eigenvector $\overline{v_2}$. Take $x = v_2 + \overline{v_2} = 2 \Re{v_2} \in \R^n$. By Remark~\ref{rem:E-and-F-from-exp-on-subs}, we can take $E$ to be the span of the basis vectors associated with eigenvalues of modulus one from the real Jordan form of $A^{-1}$. Then $x \in E$ and
\[ A^{-j} x = 2 j \Re{(\lambda^{j-1}v_1)} + 2 \Re{(\lambda^j v_2)} \]
We see that the orbit $\set{A^{-j}x}_{j=1}^\infty \subset E$ is unbounded, hence (iv) in Definition~\ref{def:exp-on-subspace} cannot hold. Thus, all eigenvalues $\la$ of $A$ with $|\lambda|=1$ have Jordan blocks of order one. Moreover, at least one eigenvalue $\lambda$ of $A$ satisfies $|\la|>1$ in light of (iii). The proof of the ``if''-direction is a simple verification of properties (i)-(iv) in Definition~\ref{def:exp-on-subspace}.
\end{proof}

It was shown in \cite{MR2264324} that eigenvalues alone do not give the complete picture of when the wavelet characterization results hold. The interaction of a dilation $A$ and a lattice $\Gamma$ has to be taken into account to get the more optimal result. In fact, the following result combining \cite[Lemma 3.2]{MR2264324} and \cite[Lemma 3.3]{MR2264324} motivates the definition of lattice counting estimate \eqref{eq:lc}, see also \cite[Lemma 2.8]{MR2746669}. 

\begin{lemma}\label{explce}
  Let $A\in \mathrm{GL}(n,\R)$ be expanding on a subspace of $\R^n$, and let $r>0$. Then $(A,\LG)$ satisfies the lattice counting estimate~(\ref{eq:lc}) for any full-rank lattice $\LG \subset \R^n$. 
\end{lemma}

We finish this section by showing that the converse of Lemma~\ref{explce} holds. Hence, the class of dilations expanding on a subspace consists precisely of those dilations for which the  lattice counting estimate holds for every choice of a lattice.

\begin{theorem}\label{lceexp} 
Suppose that $A\in \mathrm{GL}(n,\R)$ and $\abs{\det{A}}>1$. Then, $(A,\LG)$ satisfies the lattice counting estimate~(\ref{eq:lc}) for \emph{all} full-rank lattices $\LG \subset \R^n$ if and only if $A$ is expanding on a subspace.
\end{theorem}

\begin{proof} 
Lemma~\ref{explce} shows the ``if''-implication. To show the converse implication, assume that $(A,\Ga)$ satisfies \eqref{eq:lc} for all lattices $\Ga$. We will show that the properties (i)--(iii) in Proposition~\ref{thm:exp-subspace-charac} hold. 

On the contrary, suppose that (i) fails, i.e., there exists an eigenvalue $\la$ of $A$ such that $|\lambda|<1$. Then, there exists $1$-dimensional eigenspace $V$ if $\lambda$ is real, or $2$-dimensional invariant space $V$ corresponding to a pair of  complex conjugate eigenvalues $\lambda$ and $\overline \lambda$. In either case, we have $|Av|=|\lambda| |v|$ for all $v\in V$. Choose a full rank lattice $\Gamma$ in $V$ and extend it to a full rank lattice in $\R^n$. Then, 
\[
\#|\Gamma \cap A^j(\mathbf B(0,r)) | \ge \# |\Gamma \cap  A^j(V \cap \mathbf B(0,r))| =\#|\Gamma \cap V \cap \mathbf B(0,|\lambda|^jr)| \to \infty \qquad\text{as }j\to -\infty.
\]
This contradicts \eqref{eq:lc}. Hence, (i) holds and so does (ii) since $\abs{\det{A}}>1$. 

Finally, suppose that (iii) fails. That is, there exists a Jordan block of order $\ge 2$ corresponding to an eigenvalue $|\lambda|=1$.
If $\lambda$ is real, then $\lambda=\pm 1$ and there exists a $2$-dimensional invariant subspace $V$ such that $A|_V$ has a matrix representation $\begin{bmatrix} \lambda & 1 \\ 0 & \lambda \end{bmatrix}$. For simplicity assume $\lambda=1$. The case $\lambda=-1$ is similar. Then,
\begin{equation}\label{sh3}
(A|_V)^j = \begin{bmatrix} 1 & j \\ 0 & 1 \end{bmatrix}
\end{equation}
Choose $\alpha \in \R \setminus \Q$ and define a lattice $\Gamma$ in $V$ of the form 
$\Gamma = \Z(0,1)  + \Z(1,\alpha)$.
Then, for any $N\in \N$, we can find $\gamma = ( \gamma_1 , \gamma_2 ) \in \Gamma$ such that $\ga_1<0$ and $0<\gamma_2 <1/N$. Let $j=\lfloor \gamma_1/\gamma_2 \rfloor <0 $. Since the image of the unit square $[0,1]^2$ under \eqref{sh3} is a parallelogram with vertices $(0,0)$, $(1,0)$, $(j,1)$ and $(j+1,1)$, it contains line segments going through the origin with slopes $m$ such that $1/j\le   m \le 1/(j+1)$. In particular, the slope $m$ of the line $\R(\gamma_1,\gamma_2)$ lies in this range. Since $0<\gamma_2<1/N$, at least $N$ points of the lattice $\Gamma$ lie in the above parallelogram. Thus, 
\begin{equation}\label{sh4}
\sup_{j<0} \# | \Gamma \cap (A|_V)^j(\mathbf B(0,r))| = \infty.
\end{equation}
Extending the rank 2 lattice $\Gamma$ to a full rank lattice yields a pair $(A,\Gamma)$ that fails the lattice counting estimate \eqref{eq:lc} for $j<0$, which is a contradiction.

If $\lambda= e^{i\theta}$ is not real, then there exists a $4$-dimensional invariant subspace $V$ such that $A|_V$ has a matrix representation 
\[
\begin{bmatrix} R(\theta) & \mathbf I \\ \mathbf 0 & R(\theta)
\end{bmatrix},
\qquad\text{where }
R(\theta) = \begin{bmatrix} \cos \theta & \sin \theta \\ -\sin\theta & \cos \theta
\end{bmatrix}.
\]
Here, $\mathbf I$ and $\mathbf 0$ are the $2\times 2$ identity matrix
and the $2\times 2$ zero matrix, respectively. Observe that
\begin{equation}\label{sh5}
(A|_V)^j = \begin{bmatrix} R(\theta j) & j R(\theta (j-1)) \\ \mathbf 0 & R(\theta j) \end{bmatrix}.
\end{equation}
Let $\mathbf B_2$ be the unit ball in $\R^2$. Since $R(\theta)$ is a rotation matrix, the image of $\mathbf B_2 \times \mathbf B_2$ under the matrix $\eqref{sh5}$ is the same as the image of the same set under the matrix
\[
\begin{bmatrix} \mathbf I & j \mathbf I \\ \mathbf 0 & \mathbf I \end{bmatrix}.
\]
By permuting the basis elements, the above matrix consists of two blocks of the form \eqref{sh3}. By the same argument as in the real case we can find a lattice $\Gamma$ satisfying \eqref{sh4}. Again this contradicts \eqref{eq:lc} and completes the proof of Theorem~\ref{lceexp}.
\end{proof}

\section{The local integrability condition}
\label{sec:local-integr-cond}

In this section we investigate the local integrability condition (LIC)
that was originally introduced by Hern\'andez, Labate, and Weiss
\cite{HLW}. While this condition can be studied in full generality of
generalized shift-invariant (GSI) system, we restrict our attention to
wavelet systems. We show that the integrability of the Calder\'on
formula implies the LIC precisely for  pairs $(B,\Ga^*)$ satisfying the lattice counting estimate \eqref{eq:lca}. As a consequence, we extend characterization results for Parseval and dual frames to this more general (than expanding on a subspace) setting.

\begin{definition}
Let $E$ be a proper subspace of $\R^n$. Consider the following dense subspace of $L^2(\Rn)$,
\begin{equation}
\label{eq:def-D}
\mathcal{D}_E  = \setprop{f \in L^2 (\Rn)}{\hat f \in L^\infty(\Rn) \text{ and }
\overline{\supp \hat f} \subset \Rn \setminus E \text{ is compact}}.
\end{equation}

Let $\Psi = \set{\psi_1,\dots, \psi_L} \subset L^2(\R^n)$, $A\in \mathrm{GL}(n,\R)$ and $\Ga$ is a full-rank lattice. The corresponding wavelet system associated with the pair $(A,\Ga)$ is defined as
\[
\mathcal A(\Psi, A,\Gamma)= \{D_{A^j}T_\gamma \psi_\ell: j\in \Z,\gamma \in \LG,\ell=1,\dots,L\}
\]
where $D_{A}f(x)=|\det A|^{1/2} f(Ax)$ is the dilation operator and $T_\gamma f(x)=f(x-\gamma)$ is the translation operator. We say $\mathcal A(\Psi,A,\Ga)$ satisfies the local integrability condition (LIC) if
 \begin{align}
    L(f)&= \sum_{l=1}^L\sum_{j \in \Z} \sum_{\mlat \in \LG^\ast} \int_{\supp
      \hat f} \abssmall{\hat f(\xi+B^j \mlat)}^2 \, \abssmall{\det{A}}^j
    \abs{\ft \dila \psi_l(\xi)}^2\, \mathrm{d}\xi \nonumber   \\ 
&=\sum_{l=1}^L\sum_{j \in \Z} \sum_{\mlat \in \LG^\ast} \int_{\supp
      \hat f} \abssmall{\hat f(\xi+B^j \mlat)}^2 \, \abssmall{\hat
      \psi_l(B^{-j}\xi)}^2\, \mathrm{d}\xi < \infty \quad \text{for all
      $f \in \mathcal{D}_E$.}  \label{eq:af-LIC}
   \end{align}
Here, the Fourier transform is defined for $f \in L^1(\Rn)$ by 
\[
\ft
f(\xi)=\hat f(\xi) = \int_{\Rn} f(x)\mathrm{e}^{-2 \pi i
  \innerprod{\xi}{x}} \mathrm{d}x
  \]
with the usual extension to $L^2(\Rn)$.
\end{definition}

\subsection{The Calder\'o{}n condition}
\label{sec:calderon-condition}

 The following fact shows that the local integrability condition (\ref{eq:af-LIC}) for the wavelet system $\af{\Psi,A,\Ga}$ implies the local integrability
of the Calder\'o{}n sum (\ref{eq:L1-local-int}). 
\begin{lemma} \label{thm:LIC-implies-L1-loc}
 Let $A \in \mathrm{GL}(n,\R)$ and let $E$ be a proper subspace of $\R^n$. Suppose $\Psi = \set{\psi_1,\dots, \psi_L} \subset
  L^2(\Rn)$ satisfies the LIC (\ref{eq:af-LIC}) for $f \in \mathcal{D}_E$. Then 
  \begin{equation}
\sum_{l=1}^L \sum_{j\in\Z} \abs{\hat \psi_l(B^{-j}\xi)}^2  \in
L^1_{\text{loc}}(\Rn \setminus E). \label{eq:L1-local-int} \\
\end{equation}
\end{lemma}
\begin{proof}
  Suppose that $L(f) < \infty$ for all $f \in \mathcal{D}_E$. Then, in
particular by choosing $\hat f = \chi _{S}$ for a compact set $S
\subset \Rn \setminus E$, we have
\[ \int_S \sum_{l=1}^L \sum_{j \in \Z}\abs{\hat
      \psi_l(B^{-j}\xi)}^2\, \mathrm{d}\xi = \sum_{l=1}^L\sum_{j \in \Z} \int_S \abs{\hat
      \psi_l(B^{-j}\xi)}^2\, \mathrm{d}\xi \le L(f) < \infty.
\] 
Since the set $S$ was arbitrarily chosen,  the validity of (\ref{eq:L1-local-int}) follows. 
\end{proof}

\begin{definition}
We say that a Lebesgue measurable set $S \subset \R^n$ is a \emph{multiplicative tiling set} under $A \in \mathrm{GL}(n,\R)$ if
\begin{enumerate}[(a)]
\item $ \bigcup_{j \in \Z} A^j (S) = \R^n, $
\item $A^j(S) \cap A^i(S) = \emptyset \text{ whenever } j \neq i \in \Z $,
\end{enumerate}
where each equality is up to sets of measure zero. 
\end{definition}

Larson, Schulz, Speegle, and Taylor \cite[Theorem 4]{MR2249311} have shown the following interesting result about multiplicative tilings of $\R^n$.

\begin{theorem}
\label{thm:larson-schulz-speegle-cross-sections}
Let $A \in \mathrm{GL}(n,\R)$. 
\begin{enumerate}[(i)]
\item
There exists a multiplicative tiling set if and only if $A$ is not orthogonal.
\item
There exists a multiplicative tiling set of finite measure  if and only if  $\abs{\det{A}}\neq 1$. 
\item
There exists a bounded multiplicative tiling set  if and only if all
eigenvalues of $A$, in modulus, are either strictly greater or strictly smaller than $1$. 
\end{enumerate}
\end{theorem}

The following fact is a consequence of the main result of Laugesen,
Weaver, Weiss, and Wilson in \cite{lwww}. It can also be deduced from Theorem~\ref{thm:larson-schulz-speegle-cross-sections} as we see below.

\begin{lemma}
\label{thm:existence-calderon}
Let $A \in \mathrm{GL}(n,\R)$. Then $\abs{\det{A}}\neq 1$ if and only if there exists a function $\psi \in L^2(\R^n)$ such that 
\begin{equation}\label{ec1}
\sum_{j \in \Z} \abssmall{\hat \psi(B^{-j} \xi)}^2 = 1
\qquad\text{for a.e. }\xi \in \R^n.
\end{equation}
\end{lemma}
\begin{proof}
If $\abs{\det{A}} \ne 1$, then we simply take $\hat \psi = \chi_S$, where $S$ is a multiplicative tiling set of finite measure for $B=A^T$. Conversely, assume that \eqref{ec1} holds. There are two cases to consider. Suppose that $A$ is an orthogonal matrix. Then, by the change of variables formula for any $R>1$, we have
\[
|\{\xi\in \R^n: 1/R<|\xi|<R \}| = \int_{1/R<|\xi|<R} \sum_{j \in \Z} \abssmall{\hat \psi(B^{-j} \xi)}^2 d\xi = \sum_{j\in \Z} \int_{1/R<|\xi|<R} \abssmall{\hat \psi(\xi)}^2 d\xi.
\]
The left-hand side of this equation is finite and positive, while the
term of the far right is either zero or infinite, which is a contradiction.
Suppose next that $A$ is not orthogonal and $|\det A|=1$. Let $S$ be a multiplicative tiling set for $B$. Since $|\det{B}|=1$, by the change of variables formula, we have
\[
|S|= \int_S \sum_{j \in \Z} \abssmall{\hat \psi(B^{-j} \xi)}^2 d\xi = \sum_{j\in \Z} \int_{B^{-j}S} \abssmall{\hat \psi(\xi)}^2 d\xi= ||\hat \psi||^2 = ||\psi||^2.
\]
This implies that $S$ has finite measure. Then, Theorem~\ref {thm:larson-schulz-speegle-cross-sections}(ii) yields $|\det{B}| \ne 1$, which is a contradiction. Consequently, we have $|\det A| \ne 1$.
\end{proof}

\subsection{The local integrability condition and the lattice counting estimate}

The main result of this section shows a link between the lattice counting estimate, the local integrability condition, and  Calder\'on's formula. We start with a necessary definition of sets that appear in the proof of Theorem~\ref{thm:lic-and-calderon-under-volume-bound}.

\begin{definition}
  \label{def:def-of-E-and-F}
  For a given $A \in \mathrm{GL}(n,\R)$ we consider $B=A^T$ as a linear map
  acting on $\C^n$. Let $E^c \subset \C^n$ and $F^c \subset \C^n$ be the
  span of eigenspaces corresponding to eigenvalues $\lambda$ of $B$
  satisfying $\abs{\lambda}\le 1$ and $\abs{\lambda} > 1$,
  respectively. Define $E = E^c \cap \R^n$ and $F = F^c \cap
  \R^n$. For $p,q,s>0$, define
  \[
  Q(p,q,s)=\setprop{x=x_E+x_F}{x_E \in E, x_F \in F, \abs{x_E}<p, s<
    \abs{x_F} <q}.
  \]
\end{definition}
Since complex eigenvalues of $B$ come in conjugate pairs, the spaces $E^c$ and $F^c$ are complexifications of the real spaces $E$ and $F$, respectively.

\begin{lemma}
\label{rem:tiling-away-from-E}
Let $B\in \mathrm{GL}(n,\R)$ with $|\det B|>1$ be given.
For any $\eps, s>0$, there exists a multiplicative tiling set $S$ for the dilation $B$ such that for some $p,q>0$ we have
\begin{equation}\label{taf0}
\abs{S \setminus Q(p,q,s)} < \eps \abs{S}.
\end{equation}
\end{lemma}

That is, for a given $\eps>0$ and $s>0$, we can always find a multiplicative
tiling set for the dilation $B$ that, up to a relative error $\eps$, lies inside
$Q(p,q,s)$ for sufficiently large $p,q>0$.

\begin{proof}
By Theorem~\ref{thm:larson-schulz-speegle-cross-sections}, there exists a multiplicative tiling set $S_0$ for $B$ of finite measure. Since 
\[
\R^n \setminus E = \bigcup_{\delta>0} Q(\infty,\infty,\delta),
\]
we can find $\delta>0$ such that
\begin{equation}\label{taf1}
\abs{S_0 \setminus Q(\infty,\infty,\delta)} < \frac{\eps}2 \abs{S_0}.
\end{equation}
For any $j\in \N$, define
\[
S_j = B^j(S_0 \cap Q(\infty,\infty,\delta)) \cup (S_0 \setminus Q(\infty,\infty,\delta)).
\]
Clearly, $S_j$ is a multiplicative tiling set for
$B$. Since the dilation $B$ is expanding in the direction of the space $F$, there exists $j\in\N$ such that 
\begin{equation}\label{taf3}
B^j( Q(\infty,\infty,\delta)) \subset Q(\infty,\infty, s).
\end{equation}
Combining \eqref{taf1} and \eqref{taf3} yields
\[
|S_j \setminus Q(\infty,\infty,s)| < \frac{\eps}2 |S_0| \le \frac{\eps}2 |S_j|.
\]
Hence, by choosing sufficiently large $p,q>0$ we have
\[
|S_j \setminus Q(p,q,s)| < \eps |S_j|,
\]
which shows \eqref{taf0}.
\end{proof}

The following result characterizes the local integrability condition. 

\begin{theorem}
\label{thm:lic-and-calderon-under-volume-bound}
Let $A \in \mathrm{GL}(n,\R)$ with $\abs{\det{A}}>1$ be given, and let $\LG \subset \R^n$ be a full-rank lattice. The following assertions are equivalent:
\begin{enumerate}[(i)]
\item $(B,\LG^*)$ satisfies the lattice counting estimate~\eqref{eq:lca},
\item For any $\Psi = \set{\psi_1, \dots, \psi_L} \subset L^2(\R^n)$, the LIC~\eqref{eq:af-LIC} holds for $\af{\Psi,A,\LG}$ if and only if $\Psi$ satisfies the Calder\'o{}n integrability condition \eqref{eq:L1-local-int}.    
\end{enumerate}
\end{theorem}

\begin{proof}
Let $E$, $F$, and $Q(p,q,s)$, be as in Definition~\ref{def:def-of-E-and-F}.

 (i) $\Rightarrow$ (ii):  Let  $\Psi = \set{\psi_1, \dots, \psi_L} \subset L^2(\R^n)$. Suppose that the lattice counting estimate~(\ref{eq:lca}) holds.  If $\af{\Psi,A,\LG}$ satisfies the LIC, then, by Lemma~\ref{thm:LIC-implies-L1-loc}, $\Psi$ satisfies the Calder\'o{}n integrability condition. 

 Assume on the other hand that the Calder\'o{}n integrability
 condition \eqref{eq:L1-local-int} holds. For simplicity assume that $L=1$. Let $f \in \mathcal{D}_E$. Then $T:=\supp \hat f
 \subset Q(p,q,s) \subset \R^n \setminus E$ for some $s,p,q>0$. 
 By the lattice counting estimate \eqref{eq:lca}, we have
\[ \cardbig{B^j \Gamma^* \cap (\supp \hat{f}-\supp \hat{f})} \le  C
    \max{(1,\abs{\det{B}}^{-j})}. \]
Hence,
 \begin{align*}
    L(f) &\le \sum_{j \in \Z} \normsmall[\infty]{\hat f}^2 C
    \max{(1,\abs{\det{B}}^{-j})}  \int_{\supp
      \hat f} \abs{\hat \psi(B^{-j} \xi)}^2\, \mathrm{d}\xi  \\
& = \normsmall[\infty]{\hat f}^2 C \sum_{j \ge 0}   \int_{T} \abs{\hat \psi(B^{-j} \xi)}^2\, \mathrm{d}\xi 
+  \normsmall[\infty]{\hat f}^2 C \sum_{j < 0}  \int_{B^{-j}(T)} \abs{\hat \psi(\xi)}^2\, \mathrm{d}\xi .  \end{align*}
Since the matrix $B$ is expansive on $F$, there exists a constant $K \in \N$ such that each trajectory $\{B^j \xi\}_{j\in\Z}$ hits $Q(p,q,s)$ at most $K$ times. Thus,
\[ 
\cardbig{\setprop{j \in \Z}{\xi \in B^{-j}(T)}} \le K.
\]
Thereby, we can continue the above estimate:
\begin{align*}
L(f)\le \normsmall[\infty]{\hat f}^2 C    \int_{T} \sum_{j \ge 0} \abs{\hat \psi(B^{-j} \xi)}^2\, \mathrm{d}\xi 
+  \normsmall[\infty]{\hat f}^2 C K \int_{ \Rn}
\abs{\hat \psi(\xi)}^2\, \mathrm{d}\xi < \infty,
  \end{align*}
where the last inequality is a consequence of (\ref{eq:L1-local-int}) and
$\psi \in L^2(\Rn)$. 

 (ii) $\Rightarrow$ (i): We prove the contrapositive assertion. So suppose that $(B,\LG^*)$ does not satisfy the lattice counting estimate~(\ref{eq:lca}) for some $r>0$. Since the  lattice counting estimate fails for either $j \in -\N$ or $j \in \N$, we have two cases:
 \begin{enumerate}[(a)]
\item $\sup_{j < 0}w_j = \infty$, where $w_j:= \card{B^j\LG^* \cap \B(0,r)} \abs{\det{B}}^j$, 
 \item $\sup_{j \ge 0}v_j = \infty$, where $v_j:=\card{B^j\LG^* \cap \B(0,r)}$.
 \end{enumerate}

Suppose case (a) holds. Choose a subsequence $\seq{w_{j_i}}_{i=1}^\infty$ of $\seq{w_{-1},w_{-2},\dots}$ such that $0>j_1>j_2>\dots$ and 
\[ 
\sum_{i=1}^\infty \frac{1}{w_{j_i}} < \infty.
\]
Let $P:\Rn \to \Rn$ be a projection (not necessarily orthogonal) such that $\ker P = E$ and $P(\Rn)=F$.
Let $\eps>0$ and pick $s>||P||r$. Let $S$ be a
multiplicative tiling set for $B$ as in Lemma~\ref{rem:tiling-away-from-E}. 
Define $\psi:\R^n \to \C$ by
\[ \hat \psi(\xi)= \sum_{i=1}^\infty \frac{1}{\sqrt{v_{j_i}}} \chi_{B^{-j_i}(S)} (\xi).\] 
We have $\psi \in L^2(\Rn)$ since
\begin{align*}
  \int_{\R^n} \abssmall{\hat \psi(\xi)}^2 \mathrm{d}\xi = \sum_{i=1}^\infty \frac{1}{v_{j_i}} \int_{\R^n} \chi_{B^{-j_i}(S)} \mathrm{d}\xi = \abs{S} \sum_{i=1}^\infty \frac{\abs{\det{B}}^{-j_i}}{v_{j_i}} = \abs{S} \sum_{i=1}^\infty \frac{1}{w_{j_i}} < \infty.
\end{align*}
Since $\xi \mapsto \sum_{j \in \Z} \abssmall{\hat\psi (B^{-j} \xi)}^2$
is $B$-dilative periodic, we see that  
\begin{align*}
  \sum_{j \in \Z} \abssmall{\hat\psi (B^{-j} \xi)}^2 =
  \sum_{i=1}^\infty \frac{1}{v_{j_i}} < \infty \qquad \text{for
    a.e. $\xi \in S$},
\end{align*}
also holds for a.e. $\xi \in \R^n$.
Hence, the  Calder\'o{}n integrability condition \eqref{eq:L1-local-int} holds. 

We now show that $\af{\Psi,A,\LG}$ does not satisfy the LIC. Define $T=(S \cap
Q(p,q,s)) + \B(0,r)$. Since $s>||P||r$,  we claim that $\overline{T} \subset \R^n \setminus E$. Indeed, take any $x\in T$ and write it as 
\[
x=x_1+x_2, \qquad x_1 \in Q(p,q,s),\ x_2 \in \B(0,r).
\]
Then,
\[
||Px|| \ge ||Px_1||- ||Px_2|| \ge s- ||P||r>0.
\]
Hence, $\overline{T} \subset \R^n \setminus E$ is compact.

Let $\hat f = \chi_T$. Then $f \in \mathcal{D}_E$, and by definition of $T$ and $v_j$, we have for $j \in \Z$,
\[ 
\cardpropbig{ \{ \mlat \in \LG^\ast}{\hat f(\xi + B^j \mlat)=1 \text{ for } \xi
  \in S\cap Q(p,q,s) \} } \ge  v_j . 
\] 
From this and $S \cap Q(p,q,s) \subset T=\supp \hat f$, it follows that
\begin{align*}
  L(f) \ge \sum_{j <0} \sum_{\mlat \in \LG^\ast} \int_{\supp
      \hat f} \abssmall{\hat f(\xi+B^j \mlat)}^2 \, \abssmall{\hat
      \psi(B^{-j}\xi)}^2\, \mathrm{d}\xi \ge  \sum_{j <0} v_j \int_{S\cap Q(p,q,s)} \abssmall{\hat
      \psi(B^{-j}\xi)}^2\, \mathrm{d}\xi.
\end{align*}
By a change of variables and Lemma~\ref{rem:tiling-away-from-E}
\begin{align*}
  L(f) \ge &\sum_{i=1}^\infty  v_{j_i} \int_{B^{-j_i}(S\cap Q(p,q,s))} \abssmall{\hat
      \psi(\xi)}^2 \abs{\det{B}}^{j_i} \,  \mathrm{d}\xi \\ &= \sum_{i=1}^\infty \abs{S\cap Q(p,q,s) } \ge \sum_{i=1}^\infty (1-\eps)\abs{S} = \infty 
\end{align*}

Suppose now case (b) holds. Choose a subsequence $\seq{v_{j_i}}_{i=1}^\infty$ of $\seq{v_0,v_1,\dots}$ such that $0\le j_1<j_2<\dots$ and 
\[ \sum_{i=1}^\infty \frac{1}{v_{j_i}} < \infty ,\]
and define $\psi:\R^n \to \C$ by
\[ 
\hat \psi(\xi)= \sum_{i=1}^\infty \frac{1}{\sqrt{w_{j_i}}} \chi_{B^{j_i}(S)} (\xi).\] 
The rest of the proof is dealt in a similar way as in the case (a) and is left to the reader.
\end{proof}

\subsection{Characterizing equations}
\label{sec:char-equat}

The papers \cite{MR2264324} and \cite{HLW} establish wavelet characterizing equations for dilations that are expanding on a subspace. In light of Theorem~\ref{lceexp}, these are optimal results unless extra information about a lattice is also taken into account. Here we shall also show the characterizing equations for pairs of dilations and lattices $(B,\Gamma^*)$ satisfying the lattice counting estimate (\ref{eq:lca}).

The following result generalizes \cite[Theorem 6.6]{HLW}, see also \cite[Theorem
1.1]{MR2264324}. For the definitions of (Parseval) frames, dual frames and Bessel
sequences, we refer the reader to the book \cite{MR1946982}.

\begin{theorem}\label{ch}
Let $A \in \mathrm{GL}(n,\R)$, $\abs{\det{A}}>1$, and $\LG \subset \R^n$ is a full-rank lattice. 
 Suppose that $(B,\Gamma^*)$ satisfies the lattice counting estimate
 \eqref{eq:lca}.
 Then, the wavelet system
$\af{\Psi,A,\Ga}$ generated by $\Psi = \set{\psi_1, \dots, \psi_L} \subset L^2(\R^n)$ is a Parseval frame if and only if
for all $\alpha\in\Ga^*$ we have
\begin{equation}\label{ch1}
\sum_{l=1}^L \sum_{j\in \Z, B^{-j}\alpha \in \Ga^*} \hat\psi_l(B^{-j}\xi) \overline{\hat \psi_l(B^{-j}(\xi+\alpha))} = \delta_{\alpha,0}
\qquad\text{for a.e. }\xi\in \R^n.
\end{equation}
\end{theorem}

\begin{proof}
It is well-known that if a wavelet system is a Bessel sequence with bound $C>0$, then the Calder\'on formula is bounded by $C$, i.e.,
\begin{equation}\label{ch2}
\sum_{l=1}^L \sum_{j\in \Z} |\hat\psi_l(B^{-j}\xi)|^2 \le C\qquad\text{for a.e. }\xi\in \R^n.
\end{equation}
This result holds without any a priori assumptions on the dilation $A$ and the lattice $\Ga$, as it is a consequence of a more general result that holds for generalized shift-invariant (GSI) systems, see \cite[Proposition 4.1]{HLW}. Thus, if $\Psi$ is a Parseval frame, then \eqref{ch2} holds for $C=1$. Likewise, if  \eqref{ch1} holds, then by setting $\alpha=0$, we also have \eqref{ch2} for $C=1$. In either case, the LIC holds for $\af{\Psi,A,\Gamma}$ in light of Theorem~\ref{thm:lic-and-calderon-under-volume-bound}. Consequently, the general machinery of Hern\'andez, Labate, and Weiss \cite{HLW} applies. By \cite[Theorem 4.2]{HLW}, the wavelet system $\af{\Psi,A,\Gamma}$ is a Parseval frame if and only if \eqref{ch1} holds.
\end{proof}

If the wavelet system generated by $\Psi=\{\psi_1,\dots, \psi_L \}$ is a
Parseval frame, it is an orthonormal basis precisely when $\norm{\psi_\ell}=1$ for
each $\ell=1,\dots, L$. Hence, Theorem~\ref{ch} also characterizes orthonormal
wavelets. Moreover, Theorem~\ref{ch} generalizes to dual wavelet frames. Indeed,
using \cite[Theorem 9.1]{HLW}, one can easily show the generalization of
\cite[Theorem 9.6]{HLW} from the setting of dilations expanding on a subspace to
the lattice counting estimate \eqref{eq:lca}.

\begin{theorem}\label{thm:ch-dual}
 Suppose that $(B,\Gamma^*)$ satisfies the lattice counting estimate
 \eqref{eq:lca}. Suppose that the wavelet systems $\af{\Psi,A,\Ga}$  and
 $\af{\Phi,A,\Ga}$ generated by
 $\Psi = \set{\psi_1, \dots, \psi_L} \subset L^2(\R^n)$ and $\Phi = \set{\phi_1, \dots, \phi_L} \subset L^2(\R^n)$, resp., are Bessel sequences. Then they are dual frames if and only if
for all $\alpha\in\Ga^*$ we have
\begin{equation}\label{eq:ch1-dual}
\sum_{l=1}^L \sum_{j\in \Z, B^{-j}\alpha \in \Ga^*} \hat\psi_l(B^{-j}\xi) \overline{\hat \phi_l(B^{-j}(\xi+\alpha))} = \delta_{\alpha,0}
\qquad\text{for a.e. }\xi\in \R^n.
\end{equation}
\end{theorem}

Theorem~\ref{thm:lic-and-calderon-under-volume-bound} shows that the lattice counting estimate \eqref{eq:lca} is the optimal hypothesis under which one should expect the characterizing equations \eqref{ch1} to hold. Indeed, if $(B,\Ga^*)$ does not satisfy \eqref{eq:lca}, then the LIC must fail for some choice of $\Psi$ and the known techniques collapse. However, this does not completely close the problem since the LIC is merely a convenient sufficient condition for showing characterization results. In particular, the following problem raised in \cite{BR, Sp} remains open.

\begin{conjecture}\label{calderon} Suppose that wavelet system $\af{\Psi, A,\Ga}$ is an orthonormal basis, or more generally a Parseval frame
  for $L^2(\R^n)$. Then the Calder\'on sum formula holds 
\begin{equation}\label{ch7}
\sum_{l=1}^L \sum_{j\in \Z} |\hat\psi_l(B^{-j}\xi)|^2 =1\qquad\text{for a.e. }\xi\in \R^n.
\end{equation}
\end{conjecture}

Theorem~\ref{ch} implies that \eqref{ch7} is true when the lattice counting
estimate holds; in particular, the conjecture is true in one dimension. 
Moreover, this conjecture is also valid for continuous wavelets where
translates along a fixed lattice $\Ga$ are replaced by translates along $\R^n$,
see \cite[Theorem 1.1]{lwww} and \cite[Proposition 1]{MR2249311}.  In particular, by Lemma~\ref{thm:existence-calderon}, we must necessarily have $\abs{\det{A}}
 \ne 1$. For continuous wavelet systems $\{T_\gamma
 D_{A^j}\psi\}_{\gamma \in \R^n,j \in \Z, \psi \in \Psi}$ with respect to discrete
 group of dilations $\{A^j:j\in \Z\}$ and translations along $\R^n$,
 the Calder\'on sum formula
\eqref{ch7} is a necessary and sufficient condition for $\{T_\gamma
 D_{A^j}\psi\}_{\gamma \in \R^n,j \in \Z, \psi \in \Psi}$ to be a
continuous Parseval frame, see \cite[Theorem 2]{MR2249311}. 
Actually, this is true for any continuous translation
invariant system $\{T_\gamma g_p\}_{\gamma \in \R^n,p\in P}$, where
$(P,\mu_P)$ is a $\sigma$-finite measure space and $\{g_p \}_{p \in
  P}\subset L^2(\R^n)$, see \cite{JakobsenReproducing2014}. In this case, the (generalized) Calder\' on
formula is $\int_{P} \abs{\hat{g}_p(\xi)}^2 d\mu_P(p) = 1$ for
a.e. $\xi \in \R^n$. 

In contrast to the continuous case, Conjecture~\ref{calderon} remains a
surprisingly intractable problem. In particular, it is not even known whether
the existence of discrete orthonormal (or Parseval) wavelet $\Psi$ for the pair
$(A,\Ga)$ implies $\abs{\det{A}}\ne 1$. Finally,
Conjecture~\ref{calderon} is a special case of the stronger conjecture
that if the wavelet system $\af{\Psi, A,\Ga}$ is a frame for $L^2(\R^n)$ with bounds $C_1$ and
  $C_2$, then 
 $C_1 \le \sum_{l=1}^L \sum_{j\in \Z} |\hat\psi_l(B^{-j}\xi)|^2 \le
 C_2$ for
a.e. $\xi \in \R^n$.

\section{Ubiquity of the lattice counting estimate}

In this section we will show that the lattice counting estimate holds almost surely for generic choices of dilations and lattices.
By Theorem~\ref{lceexp}, for any dilation $A$ that is not expanding on a subspace, one can find a full-rank lattice $\LG$ for which the lattice counting estimate (\ref{eq:lc}) fails. 
On the other hand, we shall show in this section that the lattice counting estimate (\ref{eq:lc}) holds for any dilation $A$ with $\abs{\det{A}}>1$ for almost every choice of a lattice $\LG$. We shall establish similar results on the existence of MSF wavelets.
Our techniques rely on the work of Skriganov \cite{Sk2, Sk3} on the logarithmically small errors in the lattice point problem for polyhedra.

For $x=(x_1,\ldots, x_n) \in \R^n$, define the norm form $\Nm x= x_1x_2 \cdots x_n$. Let 
\[
\Nm \La= \inf\{\abs{\Nm x}: x \in \Lambda \setminus \{0\} \}.
\]
A lattice $\La$ is said to be \textit{admissible} if $\Nm \La>0$. For such
lattices Skriganov \cite{Sk1} has established the following asymptotic bound on the number
of lattice points inside a dilation of a parallelepiped $\Pi \subset \R^n$  with edges parallel to the coordinates axes:    
\begin{equation}\label{eq:tpi}
\#(\Lambda \cap t\Pi) = t^n\abs{\Pi}+O((\log t)^{n-1})\qquad
\text{as }t\to\infty.
\end{equation}

Let $\mathcal L_n$ be the set of unimodular lattices (with volume = 1) which can
be identified with 
\[
\mathcal L_n = \mathrm{SL}(n,\R)/\mathrm{SL}(n,\Z).
\]
Even though the subset of admissible lattices is dense in $\mathcal L_n$, it has
zero measure with respect to the invariant (probability) measure $\mu_{\mathcal L}$
on $\mathcal L_n$, see \cite{Sk2}. Hence, admissible lattices are rare, i.e.,
$\mu_{\mathcal L}$-almost surely we have $\Nm \La =0$. 

Skriganov \cite{Sk2} introduced a diophantine characteristic of a lattice $\Lambda$, which measures the rate at which $\Nm \La=0$ is achieved, defined by
\begin{equation}\label{nu}
\nu(\Lambda,\rho) = \min \{ \abs{\Nm \ga}: \ga \in \Lambda, 0<|\ga|<\rho\}, \qquad \rho>||\Lambda||:=\min\{|\ga|: \ga\in \Lambda \setminus\{0\} \}.
\end{equation}
The following result, Lemma~\ref{thm:sk}, plays a key role in showing the main result of \cite{Sk2}  which says that the bound \eqref{eq:tpi} holds when $\Pi$ is replaced by any compact polyhedron for almost every choice of $\Lambda$, albeit with a slightly worse exponent $(\log t)^{n-1+\ve}$, $\ve>0$. 

Lemma~\ref{thm:sk}  is a slight generalization of \cite[Lemma 4.3]{Sk2} due to presence of a matrix $P \in \mathrm{GL}(n,\R)$. This change corresponds to a more general norm form $x \mapsto \Nm(Px)$.

\begin{lemma}\label{thm:sk}
 Let $\La \in \mathcal L_n$ be an arbitrary lattice and let $P\in \mathrm{GL}(n,\R)$. Then for almost all orthogonal matrices $U \in \mathrm{SO}(n)$ (in the sense of the Haar measure on $\mathrm{SO}(n)$) we have
\begin{equation}\label{sk0}
\nu(P U \La, \rho) > (\log \rho)^{1-n-\ve} \qquad\text{as } \rho \to \infty,
\end{equation}
where $\ve>0$ is arbitrary.
\end{lemma}

\begin{proof}
We shall follow the proof of \cite[Lemma 4.2]{Sk2} with some necessary modifications. Suppose that $\omega:[||\Lambda||,\infty) \to (0,\infty)$ is an arbitrary monotone decreasing function satisfying
\begin{equation}\label{sk2}
\sum_{\ga\in \Lambda \setminus\{0\}} |\gamma|^{-n} \bigg( \log\frac{ |\gamma|^n}{\omega(|\gamma|)} \bigg)^{n-2} \omega(|\gamma|) <\infty.
\end{equation}

Let $\sigma$ be the unique $\mathrm{SO}(n)$-invariant measure on the unit sphere
$S^{n-1}=\{x\in \R^n: |x|=1\}$ that is normalized such that $\sigma(S^{n-1})=1$. In particular, for any $x\in S^{n-1}$ we have
\begin{equation}\label{sk8}
\sigma(V) = \mu_{SO} ( \{ U \in \mathrm{SO}(n): Ux \in V\})
\qquad\text{for any open set }
V \subset S^{n-1}.
\end{equation}
Given $P\in \mathrm{GL}(n,\R)$, define
\[
s_P(\theta)= \sigma(\{x\in S^{n-1}: |\Nm (Px)|<\theta\}).
\]
Let $\mathbf I$ be $n\times n$ identity matrix.
By the estimate (4.21) in the proof of \cite[Lemma 4.2]{Sk2} we have
\begin{equation}\label{sk1}
s_{\mathbf I}(x) < c(n) \theta \bigg(1+\log \frac1\theta \bigg)^{n-2} \qquad\text{for } 0<\theta<\frac 1{\sqrt{n}},
\end{equation}
where a positive constant $c(n)$ depends only on the dimension $n$. 

Note that the mapping $\phi: S^{n-1} \to S^{n-1}$ given by $\phi(x)= Px/|Px|$ is
a smooth diffeomorphism. Since $S^{n-1}$ is compact, the jacobian of $\phi$ is
bounded from above and below by positive constants. Thus, by the change of
variables formula, there exists a constant $c=c(P)>0$ depending on $P$ such that 
\[
\frac1c \sigma(V) \le \sigma(\phi^{-1}(V))  \le c \sigma(V)
\qquad\text{for any open set }
V \subset S^{n-1}.
\]
Consequently, we have
\begin{equation}\label{sk6}
\begin{aligned}
s_P(\theta) & \le \sigma(\{x\in S^{n-1}: |\Nm(Px)|<\theta ||P^{-1}||^n |Px|^n \})
\\
& = \sigma(\{x\in S^{n-1}: |\Nm(Px/|Px|)|<\theta ||P^{-1}||^n \}) 
\le c s_{\mathbf I}(\theta ||P^{-1}||^n).
\end{aligned}
\end{equation}
Combining \eqref{sk1} and \eqref{sk6} yields
\begin{equation}\label{sk7}
\begin{aligned}
s_P(\theta) & < c(P) c(n) ||P^{-1}||^n \theta  \bigg(1+\log \frac1{\theta ||P^{-1}||^n} \bigg)^{n-2}
\\
& < c(n,P) \theta \bigg(1+\log \frac1\theta \bigg)^{n-2}
\qquad\text{for } 
0<\theta<\frac 1{\sqrt{n}||P^{-1}||^n},
\end{aligned}
\end{equation}
where the positive constant $c(n,P)$ depends on $n$ and $P$.

For any $\gamma \in \Lambda \setminus \{0\}$ and $\theta>0$, let
\[
m_\gamma(\theta) = \mu_{SO}(\{U \in \mathrm{SO}(n): |\Nm (PU\gamma)| < \theta \}).
\]
By \eqref{sk8} and \eqref{sk7}, we have
\begin{equation}\label{sk3}
m_\gamma(\theta) = s_P\bigg( \frac{\theta}{|\ga|^n} \bigg) < c(n,P) |\ga|^{-n} \theta \bigg( 1+ \log \frac{|\ga|^n}\theta \bigg)^{n-2} \qquad\text{for }0<\theta<\frac {||\Lambda||^n}{\sqrt{n}||P^{-1}||^n}.
\end{equation}
Using \eqref{sk2} and \eqref{sk3}, the Borel-Cantelli Lemma implies that for almost all $U \in \mathrm{SO}(n)$ 
\begin{equation}\label{sk4}
|\Nm(PU\gamma)| \ge \omega(|\gamma|)
\qquad\text{for all but finitely many }\gamma \in \Lambda \setminus \{0\}.
\end{equation}
Observe also that
for every $\gamma \in \Lambda \setminus \{0\}$ we have 
\begin{equation}\label{sk9}
|\Nm(P U \gamma)|>0 \qquad\text{for almost all }U\in \mathrm{SO}(n).
\end{equation}
Note that 
\begin{equation}\label{sk5}
\begin{aligned}
\nu(P U \Lambda,\rho) 
& = \min \{ \abs{\Nm (PU\ga)}: \ga \in \Lambda, 0<|PU\ga|<\rho\}
 \\
& \le \min \{ \abs{\Nm (PU\ga)}: \ga \in \Lambda, 0<|\ga|<||P^{-1}||\rho\}.
\end{aligned}
\end{equation}
Let $\{\ga_1,\ldots,\ga_q\} \subset \Lambda$ be the exceptional set, which depends on $U$, where \eqref{sk4} fails. 
Combining \eqref{sk4}, \eqref{sk9}, and \eqref{sk5} yields $\rho_0>0$ such that
\[
\begin{aligned}
\nu(P U \Lambda,\rho) 
& \ge \min\{\omega(||P^{-1}||\rho),|\Nm(P U \gamma_1)|,\ldots, |\Nm(PU \gamma_q)| \} \\
& = \omega(||P^{-1}||\rho) \qquad\text{for }\rho>\rho_0.
\end{aligned}
\]
Since the function $\omega(\rho)=(\log \rho)^{1-n-\ve}$ with $\ve>0$ satisfies \eqref{sk2}, we obtain the bound \eqref{sk0}. This completes the proof of Lemma~\ref{thm:sk}.
\end{proof}

Skriganov's Lemma~\ref{thm:sk} plays a key role in the proof of the following main result of this section.

\begin{theorem}\label{thm:ubl}
Let $B$ be any matrix in $\mathrm{GL}(n,\R)$ with $\abs{\det{B}}>1$. Then for
any lattice $\La \in \mathcal L_n$, the pair $(B,U\La)$ satisfies the lattice counting estimate \eqref{eq:lc} for almost all (in the sense of Haar measure) $U\in \mathrm{SO}(n)$.
\end{theorem}

To prove Theorem~\ref{thm:ubl} we need two lemmas about intersection of lattices with convex symmetric bodies. The first result is the volume packing lemma which can be found in the book of Tao and Vu \cite[Lemma 3.24 and Lemma 3.26]{TaVu}.

\begin{lemma}\label{thm:vp}
Let $\Gamma \subset \R^n$ be a full rank lattice, and let $\Omega$ be a symmetric convex body in $\R^n$. Then,
\begin{equation}\label{eq:vp1}
\frac{\meas{\Omega}}{2^n \meas{\R^n/\Ga}} \le \card{\Omega \cap \Gamma}.
\end{equation}
In, addition if the vectors $\Omega \cap \Gamma$ linearly span $\R^n$, then
\begin{equation}\label{eq:vp2}
\card{\Omega \cap \Gamma} \le \frac{3^n n! \meas{\Omega}}{2^n \meas{\R^n/\Ga}}.
\end{equation}
\end{lemma}

The following lemma, which is a consequence of Minkowski's second theorem \cite[Theorem 3.30]{TaVu}, shows the existence of large proper arithmetic progressions inside $\Omega \cap \Gamma$, see \cite[Lemma 3.33]{TaVu}. 

\begin{definition}\label{def:ap}
We say that $S \subset \R^n$ is a {\it symmetric arithmetic progression of rank $s$}, if there exist $(v_1,\ldots,v_s) \in (\R^n)^s$ and $(N_1,\ldots,N_s) \in \N^s$ such that
\[
S= \{ n_1v_1 + \ldots n_sv_s: n_j \in \Z, \abs{n_j} \le N_j \text{ for all } 1\le j \le s\}.
\]
We say that such $S$ is {\it proper} if elements of $S$ are uniquely represented, or equivalently if the cardinality of $S$ equals $(2N_1+1)\cdots(2N_s+1)$.
\end{definition}

\begin{lemma}\label{thm:prog}  
Let $\Gamma \subset \R^n$ be a lattice (not necessarily of full rank), and let $\Omega$ be a symmetric convex body in $\R^n$. Then there exists a proper symmetric arithmetic progression $S$ in $\Omega \cap \Gamma$ of rank $ s\le \dim \spa(\Omega \cap \Gamma)$ such that
\[
\card{S} \ge c_n \card{\Omega \cap \Gamma},
\]
where $c_n>0$ is a universal constant which depends only on dimension $n$.
\end{lemma}

Finally, we shall need an elementary lemma on the behavior of the norm form $\Nm(x)$ under dilations.

\begin{lemma}\label{dnm}
Let $B$ be any matrix in $\mathrm{GL}(n,\R)$ with $\abs{\det{B}}>1$. Let $P\in \mathrm{GL}(n,\R)$ be such that $P^{-1}BP$ is the real Jordan form of $B$.  Then for any $\ve>0$ and $r>0$, there exists $C=C(\ve)$ such that
\begin{equation}
\label{dnm0}
|\Nm(P^{-1}x)| \le 
C |\det{B}|^{j+|j|\ve} \qquad\text{for all }x \in B^j(\B(0,r)), j\in\Z.
\end{equation}
\end{lemma}

\begin{proof} Let $J$ be a Jordan block of order $k$ corresponding to a complex eigenvalue $\lambda=a + i b$. That is,  $J$ is $(2k) \times (2k)$ matrix of the form
\[
J = \begin{bmatrix} 
R_\lambda & \mathbf I_2 & & \\
& R_\lambda & \mathbf I_2 & \\
& & \ddots & \ddots \\
& & & R_\lambda
\end{bmatrix},
\qquad\text{where } 
R_\lambda =\begin{bmatrix} a & b \\ -b & a \end{bmatrix},
\quad
\mathbf I_2 =\begin{bmatrix} 1 & 0 \\ 0 & 1 \end{bmatrix}.
\]
Then, an elementary calculation shows that there exists $C>0$ such that for all $j\in \Z \setminus\{0\}$,
\begin{equation}\label{dnm1}
|J^jy| \le 
C |j|^k |\lambda|^j |y| \qquad\text{for }y\in \R^{2k}.
\end{equation}
Thus, we have
\[
|\Nm(J^j y)| \le C^{2k} |j|^{2k^2} |\lambda|^{2kj}|y|^{2k} = C^{2k} |j|^{2k^2} |\det J|^{j}|y|^{2k} 
\qquad\text{for }j\ne0, \ y\in \R^{2k}.
\]
A similar estimate holds when $J$ is a Jordan block of order $k$ corresponding to a real eigenvalue $\lambda$, i.e.,
\[
|\Nm(J^j y)| \le C^{k} |j|^{k^2} |\det J|^{j}|y|^{k} 
\qquad\text{for }j\ne0, \ y\in \R^{k}.
\]
Since $P^{-1}B^jP$ is a block diagonal matrix consisting of such Jordan blocks, we can find a constant $C>0$ such that 
\[
|\Nm(P^{-1}B^jP y)| \le C |j|^{n^2} |\det{B}|^{j} |y|^n
\qquad\text{for }j\ne 0, \ y\in \R^{n}.
\]
Now, take any $x\in B^j(\B(0,r))$ and write it as $x=B^j y$, where $|y|<r$. Then, for any $j\ne 0$,
\[
|\Nm(P^{-1}x)|= |\Nm(P^{-1}B^jy)| \le C |j|^{n^2} |\det{B}|^{j} |P^{-1}y|^n \le
C ||P^{-1}||^n \, r^n |j|^{n^2} |\det{B}|^{j}.
\]
For any $\ve>0$, there exists $j_0$ such that $|j|^{n^2} \le |\det{B}|^{|j|\ve}$ for $|j|>j_0$. This shows \eqref{dnm0} and completes the proof of the lemma.
\end{proof}

We are now ready to prove Theorem~\ref{thm:ubl}.

\begin{proof}[Proof of Theorem~\ref{thm:ubl}] 
First, we shall show that for almost every $U\in \mathrm{SO}(n)$,
\begin{equation}\label{eq:ubl2}
\card{U\LG \cap B^j (\B(0,r))} \le C \abs{\det{B}}^j \qquad
\text{for } j\ge 0.
\end{equation}
Let $j \ge 0$. By Lemma~\ref{thm:vp}, it suffices to show that the
vectors $U\LG \cap B^j(\B(0,r))$ linearly span $\R^n$. On the
contrary, suppose they do not. By Lemma~\ref{thm:prog} there exists a
proper symmetric arithmetic progression $S$ of rank $s<n$ in $U\Gamma
\cap B^j (\B(0,r))$, see Definition~\ref{def:ap}, such that
\[
\card{S} =(2N_1+1)\cdots(2N_s+1) \ge c_n \card{U\Gamma \cap B^j (\B(0,r))} \ge \frac{c_n\meas{\B(0,r)}}{2^n \meas{\R^n/\Ga}} \abs{\det{B}}^j.
\]
Thus, there exists $1\le k \le s$ such that $N_k \ge C
\abs{\det{B}}^{j/s}$. Since $N_k v_k \in B^j (\B(0,r))$, it follows from Lemma
\ref{dnm} that for any $\ve>0$ there exists $C=C(\ve)>0$ such that
\[
\abs{\Nm(P^{-1}N_k v_k)} \le C\abs{\det{B}}^{j(1+\ve)}.
\]
Thus,
\[
\abs{\Nm(P^{-1}v_k)} \le C \abs{\det{B}}^{j(1+\ve)}/ (N_k)^n  \le C \abs{\det{B}}^{j(1+\ve-n/s)}.
\]
By choosing $\ve>0$ small enough we therefore have 
\begin{equation}\label{ubl3}
\abs{\Nm(P^{-1}v_k)} \le C \abs{\det{B}}^{-j\eta},
\qquad\text{where }
\eta =n/s-1-\ve>0.
\end{equation}
Since $v_k\in B^j(\B(0,r))$, we have $|P^{-1}v_k| \le C' \norm{B}^j$, where $C'=||P^{-1}||r$. Hence, by \eqref{nu} and \eqref{ubl3}, we have
\begin{equation}\label{ubl4}
\nu(P^{-1}U\LG, C' \norm{B}^j) \le C \abs{\det{B}}^{-j\eta}
\end{equation}
since $v_k \in U\LG$. On the other hand, Lemma~\ref{thm:sk} implies that for almost every $U\in \mathrm{SO}(n)$ we have
\begin{equation}\label{ubl5}
\nu(P^{-1} U\LG, C'\normsmall{B}^j) \ge  \left(\log( C'\normsmall{B}^j)\right)^{1-n-\ve} \ge c j^{1-n-\ve} \qquad\text{as } j\to\infty.
\end{equation}
Combining \eqref{ubl4} and \eqref{ubl5} yields a contradiction for sufficiently large $j>j_0$. Therefore, the vectors $U\LG \cap B^j((\B(0,r))$ must linearly span $\R^n$ for all $j> j_0$. Applying Lemma~\ref{thm:vp} shows \eqref{eq:ubl2} for $j>j_0$. By increasing the constant $C$ (if necessary), we obtain \eqref{eq:ubl2} for the remaining values $0\le j \le j_0$.

Next, we shall show that for almost every $U\in \mathrm{SO}(n)$, there exists $C>0$ such that
\begin{equation}\label{eq:ubl6}
\card{U\LG \cap B^j (\B(0,r))} \le C \qquad
\text{for } j< 0.
\end{equation}
Take any $0 \ne v\in U\Gamma \cap B^j (\B(0,r))$, where $j<0$. By Lemma~\ref{dnm} we have $\abs{\Nm (P^{-1}v)} \le C \abs{\det{B}}^{j(1-\ve)}$, where $\ve>0$ and $C=C(\ve)$. Since $\abs{P^{-1}v} \le \norm{P^{-1}}r\, \normsmall{B^j}\le C' \normsmall{B^{-1}}^{|j|}$, by \eqref{nu}, we have
\[
\nu(P^{-1} U\LG, C'\normsmall{B^{-1}}^{|j|}) \le C |\det{B}|^{j(1-\ve)}.
\]
On other hand, Lemma~\ref{thm:sk} implies that for almost every $U\in \mathrm{SO}(n)$, 
\[
\nu(P^{-1} U\LG, C'\normsmall{B^{-1}}^{|j|}) \ge 
\left(\log( C\normsmall{B^{-1}}^{-j})\right)^{1-n-\ve} \ge c |j|^{1-n-\ve} \qquad\text{as }j \to -\infty.
\]
Combining the last two estimates implies that $j \ge -j_0$ for some sufficiently large $j_0>0$. Therefore, the intersection 
\begin{equation}\label{ubl10}
U\LG \cap B^j(\B(0,r)) = \{0\} \qquad\text{for all }j<-j_0.
\end{equation}
By increasing constant $C$ (if necessary) we obtain \eqref{eq:ubl6}. This completes the proof of Theorem~\ref{thm:ubl}.
\end{proof}

As a consequence of Theorem~\ref{thm:ubl} and the properties of the invariant measures $\mu_\mathcal L$ from \cite[Appendix 1]{Sk2}, we have the following corollary.
 
\begin{corollary}\label{cor:ubl}
The following statements are true.
\begin{enumerate}[(i)]
\item
Let $B$ be any matrix in $\mathrm{GL}(n,\R)$ with $\abs{\det{B}}>1$. Then the pair $(B,\Gamma)$ satisfies lattice counting estimate \eqref{eq:lc} for almost all lattices $\Gamma \in \mathcal L_n$ in the sense of the invariant measure $\mu_{\mathcal L}$.
\item
Let $\Ga\subset \Rn$ be any full rank lattice. Then the pair $(B,\Gamma)$ satisfies lattice counting estimate \eqref{eq:lc} for almost every $B\in \mathrm{GL}(n,\R)$ with $|\det{B}|>1$.
\end{enumerate}
\end{corollary}

To deduce Corollary~\ref{cor:ubl} from Theorem~\ref{thm:ubl} we shall use the following lemma that is implicitly contained Skriganov's paper \cite{Sk2}.

\begin{lemma}\label{nul}
Suppose that for any lattice $\Lambda \in \mathcal L_n$, a certain property holds for lattices of the form $U \Lambda$ for almost all $U\in \mathrm{SO}(n)$ in the sense of Haar measure $\mu_{SO}$. Then, the same property holds for almost all lattices $\Lambda \in \mathcal L_n$ in the sense of the invariant measure $\mu_{\mathcal L}$.
\end{lemma}

\begin{proof} 
The proof follows along the lines of the argument by Skriganov in \cite[Lemma 4.5]{Sk2} using the fact the measure $\mu_\Lambda$ on $\La_n$ can be identified with a product measure 
\[
\mu_\Lambda = \mu_{\mathcal F} \times \mathcal \mu_{SO}.
\]
More precisely, following \cite[Appendix 1]{Sk2} consider the  quotient spaces
\begin{align*}
\mathcal H_n & =\mathrm{SO}(n) \backslash \mathrm{SL}(n,\R), \\
\mathcal F_n & = \mathcal H_n/\mathrm{SL}(n,\Z) = \mathrm{SO}(n)\backslash \mathcal L_n.
\end{align*}
We regard $\mathcal H_n$ as a homogeneous space of the group $\mathrm{SL}(n,\R)$ and $\mathcal F_n \subset \mathcal H_n$ as a fundamental set of the discrete subgroup $\mathrm{SL}(n,\Z) \subset \mathrm{SL}(n,\R)$. Then, $\mathcal H_n$ admits the unique $\mathrm{SL}(n,\R)$-invariant measure $\mu_{\mathcal F}$ normalized so that $\mu_{\mathcal F}(\mathcal F_n)=1$. Moreover, the space $\mathcal H_n$ can be identified as a submanifold in $\mathcal K$ 
\[
\mathcal H_n = \{A \in \mathcal K: \det{A} = 1\},
\]
where $\mathcal K$ is the open cone of all $n\times n$ real symmetric matrices.

For any lattice $\Lambda=P \Z^n \in \mathcal L_n$, the polar decomposition of $P\in \mathrm{SL}(n,\R)$ yields
\[
P=V A^{1/2}, \qquad\text{where }V\in \mathrm{SO}(n),\ A=P^TP \in \mathcal F_n.
\]
By \cite[(13.14)]{Sk2}, we have the following product formula for $\psi \in L^1(\mathcal L_n,\mu_{\mathcal L})$
\begin{equation}\label{weyl}
\int_{\mathcal L_n} \psi(\Lambda) d\mu_{\mathcal L}(\Lambda) = \int_{\mathcal F_n} \int_{\mathrm{SO}(n)} \psi(V A^{1/2} \Z^n) d\mu_{SO}(V) d\mu_{\mathcal F}(A),
\end{equation}
where $\mu_{SO}$ is the normalized Haar measure on $\mathrm{SO}(n)$. 

Define a function $\psi(\Lambda)=1$ when a certain property holds for $\Lambda \in \mathcal L_n$, and $\psi(\Lambda)=0$ otherwise. By our hypothesis for all symmetric positive matrices $A\in\mathcal F_n$ we have
\[
\psi(U A^{1/2} \Zn)=1 \qquad\text{for all } U \in \mathrm{SO}(n) \setminus \mathcal E_A,
\]
where the exceptional set $\mathcal E_A \subset \mathrm{SO}(n)$ has measure $\mu_{SO}(\mathcal E_A)=0$. Define the exceptional set as
\[
\mathcal E = \{ \Lambda \in \mathcal L_n: \psi(\Lambda)=0 \} =
\{ \Lambda= U A^{1/2} \Zn \in \mathcal L_n: A\in \mathcal F_n, \ U \in\mathcal E_A\}.
\]
Then, by \eqref{weyl}
\[
\mu_{\mathcal L}(\mathcal E)= \int_{\mathcal L_n} \psi(\Lambda) d\mu_{\mathcal L}(\Lambda) = \int_{\mathcal F_n} \mu_{SO}(\mathcal E_A) d\mu_{\mathcal F}(A) = 0.
\]
This completes the proof of Lemma~\ref{nul}.
\end{proof}

\begin{proof}[Proof of Corollary~\ref{cor:ubl}]
Part (i) is an immediate consequence of Theorem~\ref{thm:ubl} and Lemma~\ref{nul}. To show part (ii) we consider the exceptional set
\[
\mathcal E= \{(B,\Gamma) \in \mathrm{GL}(n,\R) \times \mathcal L_n:
\text{\eqref{eq:lc} fails for } (B,\Gamma),\ |\det{B}|>1 \}.
\]
By part~(i) each section 
\[
\mathcal E_B=\{\Gamma \in\mathcal L_n: (B,\Ga) \in \mathcal E\}
\]
has measure $\mu_{\mathcal L}(\mathcal E_B)=0$. Thus, by Fubini's Theorem $(\mu_{GL} \times \mu_{\mathcal L}) (\mathcal E)=0$, where $\mu_{GL}$ is the Haar measure on $\mathrm{GL}(n,\R)$. Consequently, for almost every lattice $\Lambda \in \mathcal L_n$ we have
\begin{equation}\label{zer}
\mu_{GL}(\mathcal E^\Lambda)=0, \qquad\text{where }
\mathcal E^\Lambda=\{ B \in \mathrm{GL}(n,\R): (B,\Lambda) \in \mathcal E\}.
\end{equation}
Observe that the lattice counting estimate \eqref{eq:lc} holds for $(B,\LG)$ if and only if it holds for $(P^{-1}BP,P^{-1}\LG)$ for any $P \in \mathrm{GL}(n,\R)$. Given any $\Lambda \in \mathcal L_n$, take $P\in \mathrm{SL}(n,\R)$ such that $\Lambda=P^{-1}\Gamma$. Since $\mathcal E^\Gamma= P \mathcal E^\Lambda P^{-1}$ and the Haar measure on $\mathrm{GL}(n,\R)$ is unimodular, we have $\mu_{GL}(\mathcal E^\Gamma)=\mu_{GL}(\mathcal E^\Lambda)$. Choosing $\Lambda \in \mathcal L_n$ such that \eqref{zer} holds, yields $\mu_{GL}(\mathcal E^\Gamma)=0$. This completes the proof of the corollary.
\end{proof}

As a corollary of Theorem~\ref{thm:ubl} and \cite[Theorem 2.5]{IW} we deduce the ubiquity of MSF wavelets with respect to random choices of dilations and lattices. 

\begin{theorem}\label{umsf}
The following statements are true.
\begin{enumerate}[(i)]
\item
Let $A$ be any matrix in $\mathrm{GL}(n,\R)$ with $\abs{\det{A}}>1$ and let $\La\subset \Rn$ be any full rank lattice. Then there exists an MSF wavelet associated with $(A,U\La)$ for almost every (in the sense of Haar measure) $U\in \mathrm{SO}(n)$.
\item
Let $A$ be any matrix in $\mathrm{GL}(n,\R)$ with $\abs{\det{A}}>1$. Then there exists an MSF wavelet associated with $(A,\Ga)$ for almost all unimodular lattices $\Gamma \in \mathcal L_n$ in the sense of the invariant measure $\mu_{\mathcal L}$.
\item
Let $\Ga\subset \Rn$ be any full rank lattice. Then there exists an MSF wavelet associated with $(A,\Ga)$ for almost every $A\in \mathrm{GL}(n,\R)$.
\end{enumerate}
\end{theorem}

\begin{proof}
To prove part (i) let $\Ga=\Lambda^*$.
The proof of Theorem~\ref{thm:ubl} shows that for some sufficiently
large $j_0=j_0(U,\Ga,r)>0$, the trivial intersection property
\eqref{ubl10} holds for all $r>0$ and for a.e. $U \in
\mathrm{SO}(n)$. In particular, for any $r>0$, there are infinitely
many $j\in \N$ such that $B^{-j}(\B(0,r/2))$ packs translationally by
$U\Lambda^*$. By \cite[Theorem 2.5]{IW}, there exists a set $W\subset
\R^n$ such that $W$ tiles $\R^d$ multiplicatively by $B$ and
translationally by $U\La^*$. In other words, $W$ is a wavelet set associated with the dilation $B$ and the lattice $U\La^*$. Thus, $\psi \in L^2(\Rn)$, defined by $\hat \psi=|W|^{-1/2} \ch_{W}$, is an MSF wavelet associated with $(A,U\La)$, where $B=A^T$ and $(U\La)^*=U\La^*$. This shows part (i).

Part (ii) follows then from Lemma~\ref{nul}. To show part (iii) observe that 
\[
\mu_{GL}(\{A \in \mathrm{GL}(n,\R): \abs{\det{A}}=1 \})=0,
\]
so it is enough to show the existence of MSF wavelets for almost every $A\in \mathrm{GL}(n,\R)$ with $\abs{\det{A}}>1$. Then (iii) is deduced from (ii) along the same lines as the proof of Corollary~\ref{cor:ubl}(ii) using the observation that there exists an MSF wavelets associated with $(A,\LG)$ if and only if it exists for $(P^{-1}AP,P^{-1}\LG)$ for any $P \in \mathrm{GL}(n,\R)$. 
\end{proof}

\bibliographystyle{amsplain}

\begin{thebibliography}{99}

\bibitem{MR2746669}
M.~Bownik, J.~Lemvig,
{\it Affine and quasi-affine frames for rational dilations}, Trans. Amer. Math. Soc. {\bf 363} (2011), no. 4, 1887--1924.

\bibitem{BR}
M. Bownik, Z. Rzeszotnik, 
{\it The spectral function of shift-invariant spaces on general lattices.} Wavelets, frames and operator theory, 49--59, Contemp. Math., 345, Amer. Math. Soc., Providence, RI, 2004.

\bibitem{BS}
M.~Bownik, D.~Speegle, 
{\it The wavelet dimension function for real dilations and dilations admitting non-MSF wavelets},
Approximation theory, X (St. Louis, MO, 2001), 63--85,
Innov. Appl. Math., Vanderbilt Univ. Press, Nashville, TN, 2002.

\bibitem{MR1946982}
O.~Christensen, \emph{An introduction to frames and {R}iesz bases}, Applied
  and Numerical Harmonic Analysis, Birkh\"auser Boston Inc., Boston, MA, 2003.


\bibitem{DLS1}
X. Dai, D. Larson, D. Speegle, {\it Wavelet sets in $\mathbb R^n$}, J. Fourier Anal. Appl. {\bf 3} (1997), no. 4, 451--456. 

\bibitem{DLS2}
X. Dai, D. Larson, D. Speegle, {\it Wavelet sets in $\mathbb R^n$. II.} Wavelets, multiwavelets, and their applications (San Diego, CA, 1997), 15--40, Contemp. Math., 216, Amer. Math. Soc., Providence, RI, 1998.

\bibitem{MR2264324}
K.~Guo and D.~Labate,
{\it Some remarks on the unified characterization of reproducing systems}, Collect. Math. {\bf 57} (2006), no. 3, 295--307.

\bibitem{HLW}
E. Hern\'andez, D. Labate, G. Weiss, {\it A unified characterization of reproducing systems generated by a finite family. II}, J. Geom. Anal. {\bf 12} (2002), no. 4, 615--662.

\bibitem{IW}
E. Ionascu, Y. Wang, {\it Simultaneous translational and multiplicative tiling and wavelet sets in $\R^2$}, Indiana Univ. Math. J. {\bf 55} (2006), no. 6, 1935--1949.

\bibitem{JakobsenReproducing2014}
M.~S. Jakobsen and J.~Lemvig,
{\it Reproducing formulas for generalized translation invariant systems on
  locally compact abelian groups},
Trans. Amer. Math. Soc., to appear, {\tt arXiv 1405:4948}.

\bibitem{Ku}
G. Kutyniok, {\it The local integrability condition for wavelet frames}, J. Geom. Anal. {\bf 16} (2006), no. 1, 155--166.

\bibitem{MR2249311}
D.~Larson, E.~Schulz, D.~Speegle, and K.~F. Taylor,
{\it Explicit cross-sections of singly generated group actions} in {\em Harmonic analysis and applications}, Appl. Numer. Harmon. Anal., pages 209--230. Birkh\"auser Boston, Boston, MA, 2006.

\bibitem{La}
R.~S.~Laugesen, 
{\it Translational averaging for completeness, characterization and oversampling of wavelets},
Collect. Math. {\bf 53} (2002), no. 3, 211--249. 

\bibitem{lwww}
R.~S.~Laugesen, N.~Weaver, G.~Weiss, E.~Wilson, {\it A characterization of the higher dimensional groups associated with continuous wavelets},
J. Geom. Anal. {\bf 12} (2002), no. 1, 89--102. 

\bibitem{Sk1}
M.~M.~Skriganov, {\it
Constructions of uniform distributions in terms of geometry of numbers}, Algebra i Analiz {\bf 6} (1994), no. 3, 200--230.

\bibitem{Sk2}
M.~M.~Skriganov, {\it Ergodic theory on ${\rm SL}(n)$, Diophantine approximations and anomalies in the lattice point problem}, Invent. Math. {\bf 132} (1998), no. 1, 1--72.

\bibitem{Sk3}
M.~M.~Skriganov, A.~N.~Starkov, {\it On logarithmically small errors in the lattice point problem}, Ergodic Theory Dynam. Systems {\bf 20} (2000), no. 5, 1469--1476.

\bibitem{Sp}
D. Speegle, 
{\it On the existence of wavelets for non-expansive dilation matrices}, 
Collect. Math. {\bf 54} (2003), no. 2, 163--179. 

\bibitem{TaVu}
T.~Tao, V.~Vu, {\it Additive combinatorics}. Cambridge Studies in Advanced Mathematics, 105. Cambridge University Press, Cambridge, 2006.

\bibitem{Wa}
Y.~Wang, {\it Wavelets, tiling, and spectral sets}, Duke Math. J. {\bf 114} (2002), no. 1, 43--57.

\end{thebibliography}

\end{document}